\documentclass{amsart} 
\usepackage{latexsym,amssymb,graphicx,amscd,amsmath}
\pdfoutput=1       
\newtheorem{thm}{Theorem}[section]
\newtheorem{lem}[thm]{Lemma}
\newtheorem{prop}[thm]{Proposition}
\newtheorem{cor}[thm]{Corollary} 
\newtheorem{de}[thm]{Definition} 
\newtheorem{rem}[thm]{Remark}

\newcommand{\BS}{{\mathcal{S}_p}}
\newcommand{\BSplus}{{\mathcal{S}^+_p}}

\newcommand{\BZ}{{\mathbb{Z}}}
\newcommand{\BQ}{{\mathbb{Q}}}
\newcommand{\BO}{{\mathcal{O}_p}}
\newcommand{\BOplus}{{\mathcal{O}_p^+}}

\newcommand{\BB}{{\mathcal{B}}}

\newcommand{\BG}{{\mathcal{G}}}

\newcommand{\BJ}{{\mathcal{J}}}

\newcommand{\BF}{{\mathbb{F}}}

\newcommand{\BSs}{{{\mathcal{S}}_p^{\sharp}}}
\newcommand{\BSps}{{{\mathcal{S}}_p^{+\sharp}}}

\newcommand{\Si}{{\Sigma}}

\newcommand{\bb}{{\mathfrak{b}}}

\newcommand{\bg}{{\mathfrak{g}}}

\newcommand{\la}{{\lambda}}

\begin{document}
\title{Integral TQFT for a one-holed torus}
 
\author{ Patrick M. Gilmer}
\address{Department of Mathematics\\
Louisiana State University\\
Baton Rouge, LA 70803\\
USA}
\email{gilmer@math.lsu.edu}
\thanks{The first author was partially supported by  NSF-DMS-0604580, NSF-DMS-0905736 }
\urladdr{www.math.lsu.edu/\textasciitilde gilmer/}

\author{Gregor Masbaum}
\address{Institut de Math{\'e}matiques de Jussieu (UMR 7586 du CNRS)\\
Universit{\'e} Paris 7 (Denis Diderot) \\
Case 7012 - Site Chevaleret\\
75205 Paris Cedex 13\\
FRANCE }
\email{masbaum@math.jussieu.fr}
\urladdr{www.math.jussieu.fr/\textasciitilde masbaum/}

\date{This version: February 28, 2011. First version: August 19, 2009}

\begin{abstract} We give 
new 
explicit
formulas for the representations of
  the mapping class group of a 
genus one surface with one boundary component
which arise from 
Integral TQFT. 
Our formulas allow one to compute 
the $h$-adic expansion of the
TQFT-matrix associated to a
mapping class in a straightforward way. 
Truncating 
the $h$-adic
expansion gives an approximation of the representation
by representations into finite groups. As a special case, we study the induced representations over finite
fields 
and identify them up to isomorphism. 
 The key technical ingredient of the paper are new bases of the
Integral TQFT modules which are orthogonal with respect to 
the
Hopf pairing. We construct these orthogonal bases in arbitrary genus, and briefly describe some other applications of them.
\end{abstract}

\maketitle
\tableofcontents

\section{Introduction} \label{sec.intro}

Consider the integral TQFT studied in \cite{G, GM}.  Associated to a
compact oriented surface $\Si$   and an odd prime $p$,
there is a free lattice ({\em i.e.,} a free finitely generated module) $\BSplus(\Si)$ over the
ring of cyclotomic integers $\BOplus=\BZ[\zeta_p]$ formed by adjoining
to $\BZ$ a primitive $p$-th root of unity $\zeta_p$. Tensoring
$\BSplus(\Si)$ with the cyclotomic field $\BQ(\zeta_p)$ yields a
version of the Reshetikhin-Turaev TQFT
associated with the Lie group $SO(3)$, and we think of $\BSplus$ as an
integral refinement of that theory (see \cite{GM} for more details). 

There is a natural representation 
$\rho_p$ 
of  
a certain central extension of $\Gamma_{\Si} $,
the mapping class group of  $\Si$, on the lattice $\BSplus(\Si)$.
Let $h$ denote $1-\zeta_p$; this is a prime in $\BOplus$. For every $N\geq 0$,
the representation
$\rho_p$ induces a representation $\rho_{p,N}$ 
on $${\mathcal{S}}^+_{p,N}(\Si)= \BSplus(\Si)\slash h^{N+1}
\BSplus(\Si)~,$$ which is a free module over the 
 quotient
 ring $\BOplus
\slash h^{N+1} \BOplus$. Note that for $N=0$ this ring is the finite
field $\BF_p$, so that $\rho_{p,0}$ is a representation on a
finite-dimensional $\BF_p$-vector space. The representation
$\rho_{p,0}$ factors through the mapping class group $\Gamma_{\Si}$ \cite{GMnew},
but this is not the case for $\rho_{p,N}$ when $N>0$.

Loosely speaking, we refer to the sequence of representations
$\rho_{p,N}$ as the $h$-adic expansion\footnote{Here, we mean the $h$-adic expansion at a fixed
  $p$-th root of unity, not to be confused with an Ohsuki-style expansion
  when $p$ goes to infinity (hence $q=\zeta_p$ goes to $1$) which  is discussed in \cite{Minprogress}.}   of the TQFT representation
$\rho_p$. 
Note that each $\rho_{p,N}$  
takes values in 
a finite group, 
since ${\mathcal S}^+_{p,N}(\Si)$ is
of finite rank over $\BOplus
\slash h^{N+1} \BOplus$, which itself is finite.  
Thus the $h$-adic expansion approximates the TQFT-representation
$\rho_p$ by representations into finite groups.  
Moreover, the natural
map $$\BOplus = \BZ[\zeta_p] \, \longrightarrow \,\,\underleftarrow{\lim} \,\, \BOplus
\slash h^{N+1} \BOplus \, \cong \BZ_p[\zeta_p]$$ is injective. (Here
$\BZ_p$ denotes the $p$-adic integers.)  
  The last isomorphism comes from
the fact that $(h^{p-1})=(p)$ as ideals in $\BZ[\zeta_p]$.  
Thus the $h$-adic expansion of the TQFT-representation
$\rho_p$ is complete in the sense that 
if a mapping class is detected by $\rho_p$ then it is also 
detected by $\rho_{p,N}$ for all
big enough $N$.

The image of $\rho_p$ is usually infinite \cite{F, Minfin}.  In this
case,   the 
 images of $\rho_{p,N}$ are finite groups whose size must increase as
$N\rightarrow \infty$.

We believe these representations into finite groups deserve further study. Unfortunately,
the known formulas for $\rho_p$ are not well suited to
computing the $h$-adic
expansion explicitly in practice, and new techniques are needed. 
In this paper we complete the first
step in this direction. 
That is: we  compute explicit matrices for
the $h$-adic expansion in the case when $\Si=\Si_{1,1}$ is a genus
one surface with one boundary component. Our formulas give partial information
in higher genus as well, but as we don't yet have a complete
answer, we won't discuss formulas for the $h$-adic expansion in the higher genus case in this paper.

We refer to  $\Si_{1,1}$ as a
one-holed torus.  As is customary in the
TQFT world, when dealing with surfaces with boundary, in practice one
caps off the boundary circles  by  disks containing
one colored banded point each in their interiors. (A banded point is a
small oriented interval.) When doing this for $\Si_{1,1}$, 
we denote the resulting closed torus by ${\mathcal T}_c$,
where $2c$ is the color of the banded point;  here $2c$ can be any
even integer satisfying $0\leq 2c\leq p-3$. Mapping classes  have to preserve the banded
point, so that the Dehn twist $t_\delta$ about a simple closed curve
$\delta$ encircling the colored banded point is non-trivial. Thus we
are indeed dealing with $\Gamma_{1,1}$, the mapping class group of the
one-holed torus. This group has
presentation
\begin{equation} \label{presGamma} \Gamma_{1,1}=<t_\alpha, t_\beta \, | \, t_\alpha t_\beta t_\alpha =
t_\beta t_\alpha t_\beta >
\end{equation}
 where $t_\alpha$ and $t_\beta$ are Dehn twists $t_\alpha$ and $t_\beta$ about essential simple
closed curves $\alpha$ and $\beta$ which avoid the colored banded
point and intersect exactly once. One
has  \begin{equation} \label{tdelta} t_\delta = ( t_\alpha t_\beta)^6~. 
\end{equation}

 We denote by  $t$ and $t^*$ the endomorphisms of
$\BSplus({\mathcal T}_c)$ given as $\rho_p$ of certain lifts  
of $t_\alpha$ and $t_\beta$ to the extended
mapping class group 
\cite{GMnew}. The
notation $t^*$ is motivated by the fact that $t$ and $t^*$ are adjoint
with respect to a symmetric bilinear form  $((\ ,\,))$ called the Hopf
pairing which will play a crucial role in this paper. To describe the
TQFT representation $\rho_p$ in our situation, it suffices to write
down matrices for $t$ and $t^*$.  

In the TQFT defined over the cyclotomic {\em field} $\BQ(\zeta_p)$, it
is well-known how to do this, because the endomorphism $t$ can be
diagonalized over $\BQ(\zeta_p)$ using a certain basis of 
`colors',
and the entries of $t t^* t$ in this basis can 
then be computed from an
appropriate (generalized) $S$-matrix. However this basis in which $t$
is diagonal is {\em not}
an  $\BOplus$-basis for the lattice $\BSplus({\mathcal T}_c)$. For example,  the
coefficients of the 
matrix of $t^*$ 
in this basis do not lie in
$\BOplus$. Hence we cannot compute the $h$-adic expansion of the
representation $\rho_p$  using this
basis. 

On the other hand, for an arbitrary surface $\Si$ equipped with
colored banded points, a $\BOplus$-basis for the lattice $\BSplus(\Si)$
was constructed explicitly in our paper \cite{GM}. In this {\em
  lollipop 
basis}, all 
mapping classes are represented by matrices with coefficients in
$\BOplus$. However, the explicit formulas we get from \cite{GM} are
still not good enough to compute the $h$-adic expansion in a nice way. 

In the present
paper, we will describe a 
modification of the construction 
of \cite{GM}. The crucial additional property of the bases obtained here
is that they are orthogonal with respect to a Hopf-like pairing $((\
,\,))$ which generalizes the one mentioned above in the genus one
case.   
As we'll see, this orthogonality property is very useful as one can
easily express general vectors in terms of the orthogonal basis, using
the pairing.  
We refer to this new basis as the {\em orthogonal lollipop
  basis.} We perform its construction in the general situation of  a surface $\Si$ equipped with
colored banded points, as the higher genus case presents no extra
difficulty, 
and give two immediate applications (see Remarks \ref{appl1} and \ref{appl2}).  
Moreover, \cite{G3} and work in progress \cite{Minprogress,GM3} make use of
this basis in 
higher genus as well.

Let us end this introduction with a brief description of how the TQFT-represen\-tation $\rho_p$
looks like in the orthogonal lollipop basis for the one-holed torus, and discuss an application we find interesting. The lattice $\BSplus({\mathcal T}_c)$ has rank $d-c$, where
$d=(p-1)/2$. The basis we'll construct is denoted by 
$$Q_n^{\prime (c)}\ (0\leq  n\leq d-c-1)~.$$ One of its main features
is that $t$ is upper triangular and $t^*$ is lower triangular in
this basis, as will be shown in Section~\ref{ptorus}. This fact is also the main reason why the new basis is better than the
one described in \cite{GM}, where only
one of $t$ and $t^*$ would be given by triangular matrix.

The matrix coefficients of
$t$ and $t^*$ in this basis will be denoted by $a^{(c)}_{m,n}$ and
$b^{(c)}_{m,n}$, respectively, and we will give explicit formulas for
them in Section~\ref{ptorus}. These formulas can be used to compute
the $h$-adic 
expansion of $a^{(c)}_{m,n}$ and
$b^{(c)}_{m,n}$; in particular, they make it evident that $a^{(c)}_{m,n}$ and
$b^{(c)}_{m,n}$ lie in $\BZ[\zeta_p]$. 

The fact that $t$ is upper triangular and $t^*$ is lower triangular
is reminiscent of the action of  $SL(2,\BZ)$ on homogeneous
polynomials in two variables. Motivated by this observation, we will show in
Section~\ref{sec5} that
the representation $\rho_{p,0}$ on the $\BF_p$-vector space
${\mathcal S}^+_{p,0}({\mathcal T}_c)$ factors through $SL(2, \BF_p)$ and is, in
fact, isomorphic to the representation of  $SL(2, \BF_p)$
on the space of  homogenous polynomials over $\BF_p$ in two variables of the
appropriate degree $D=d-c-1$. 

In order to put this last result into perspective, we remark that when
$c=0$, it is known  \cite{LW,FK,G2} that the representation
$\rho_p$ itself factors through (a finite central extension of) $SL(2,
\BF_p)$, as the situation then reduces to the classical case of the
closed 
torus. In contrast, when $c>0$, there is no
reason to believe that the representation should factor through  $SL(2,
\BF_p)$. Indeed, we 
conjecture 
that the representation $\rho_p$ 
does {\em not} factor through any finite group, although, of
course, each $\rho_{p,N}$ does,
as was explained above. 
It would follow that
the
quotient group of the extended mapping class group which acts effectively under the representation
$\rho_{p,N}$ must get bigger and bigger as $N$ increases, and one of
our motivations for this paper was to develop techniques to compute
the representations $\rho_{p,N}$ explicitly and efficiently. 

The contrast between the cases $c=0$ and $c>0$ are, of course,
expected to be related to geometry: a closed torus is Euclidean,
whereas a torus with non-empty boundary is hyperbolic. 
We remark that 
it is shown in unpublished work of the second author \cite{M2}
that 
for $c=d-2$, any pseudo-Anosov mapping class $\varphi$ 
in 
$\Gamma_{1,1}$ is represented by a matrix of infinite order under
$\rho_p$ if $p$ is big enough.
 (See \cite{AMU} for a proof of the analogous result for 
the mapping class group of a four-holed sphere.)  
In this situation, it follows that the order of $\rho_{p,N}(\varphi)$ must go to infinity as $N$
increases, showing in particular that the above-mentioned conjectural picture is correct at least in
this case.

\section{The orthogonal lollipop basis for a one-holed torus}\label{sec2}
 Let $p\geq 5$ be a prime, and $\zeta_p$ be a primitive  $p$-th root
 of unity. We denote the cyclotomic ring $\BZ[\zeta_p]$ by
 $\BOplus$. We also let $\BO$ denote $\BOplus$ if $p \equiv 3
 \pmod{4}$ and 
 $\BOplus[i]$ if $p \equiv 1 \pmod{4}$. Throughout the
 paper, we use the notation $h=1-\zeta_p$ and $d=(p-1)/2$. When
 referring to quantum integers or to skein-theoretical constructions, we will often write $q$
 for $\zeta_p$, and we put $A= - q^{d+1}$; this is a primitive $2p$-th
 root of unity, and a square root of $q=\zeta_p$. 

Here is some more notation that we will need. For $0 \le n < p$, let
$\{n\}=(-A)^n-(-A)^{-n}$,
 $\{n\}^+ = (-A)^n+(-A)^{-n}$. 
 (Warning: the minus signs make this different notation than commonly used.)
 We also  let $\{n\}_q=q^n-q^{-n}$.
As usual, we put $\{n\}! = \{1\}\{2\}\cdots\{n\}$, and  $\{0\}!=1$. If $n$ is negative, let $\{n\}!=0$. Also let
$\{n\}!! = \{n\} \{n-2\} \cdots \text{ending in $\{1\}$ or $\{2\}$}$.  
Interpret $\{n\}^+!$,  $\{n\}^+!!$, $\{n\}_q!$ similarly.

We refer the
 reader to section 2 of \cite{GM} for the definitions of the TQFT-module
 $V_p(\Si)$, defined over 
 $\BO[h^{-1}]$,  and the integral TQFT-module
 $\BS(\Si)$, defined over $\BO$. 
   In this paper,
 we are mainly interested in the refined integral TQFT-module $\BSplus(\Si)$,
 with coefficients in $\BOplus$; it is defined in \cite[Section
 13]{GM}.  Here $\Si$ is a
 compact oriented surface equipped with a (possibly empty) set of
 colored banded points. The set of
 allowed colors is $\{0,1, \cdots, p-2 \}$, but the sum of the colors
 of all the banded points must be even.  The modules  $V_p(\Si)$ and
 $\BS(\Si)$ are obtained from  $\BSplus(\Si)$ by tensoring with 
 $\BO[h^{-1}]$ 
 and $\BO$, respectively;  thus, all three modules are free of the same finite rank 
 over their respective coefficient rings, and  we have natural
 inclusions $$\BSplus(\Si) \subset \BS(\Si) \subset V_p(\Si)$$ coming
 from the inclusions $\BOplus \subset \BO\subset  
 \BO[h^{-1}]$. The $V_p$-theory is a version of the Reshetikhin-Turaev
 $SO(3)$-TQFT,  
the only difference being that the latter is usually considered with
coefficients in the quotient field of  
$\BO$
(or even with coefficients in $\mathbb C$.)

Let  ${\mathcal T}_c$ denote a torus with one banded point colored
$2c$. If $c=0$, we can forget the banded point, and we simply write
${\mathcal T}$ for ${\mathcal T}_0$.  Elements of $V_p({\mathcal
  T}_c)$ are represented skein-theoretically as linear combinations of
colored banded graphs in a solid torus, which meet the boundary nicely
in the colored banded point. See \cite{BHMV2} for the
skein-theoretical construction of TQFT we are using here. 

The {\em small graph basis}
\cite[Prop. 3.2]{GM} of  $V_p({\mathcal
  T}_c)$ is $L_{c,n}\  (0\leq n\leq d-c-1)~,$ where
\begin{equation}\label{Lbasis}L_{c,n}\ = \  \hskip.1in
 \begin{minipage}{0.4in}\includegraphics[width=0.4in]{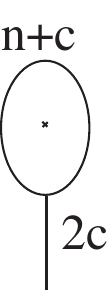}\end{minipage}
\end{equation} 
is a lollipop graph with stick color $2c$, and color $c+n$
on the loop edge. 
The {\em Hopf pairing} is the 
 non-degenerate
 bilinear form  $(( \ ,\, ))$ on $V_p({\mathcal
  T}_c)$ defined as follows: \[((L_{c,n},L_{c,m}))= \left<
  \ \ \begin{minipage}{.8in}\includegraphics[width=.8in]
  {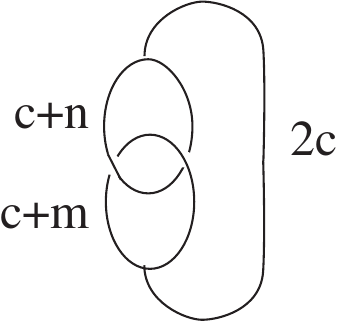}
 \end{minipage}\
\ \ \right>, \]
 where $\langle \  \rangle$ denotes the 
evaluation of a colored banded graph in $S^3$, see \cite{MV,KL}.  
(There is a choice of the type of clasp. We insist that
the clasp be as drawn. 
This is just a convention;  our formulas would differ slightly if we chose the other convention.) The Hopf pairing restricts to a symmetric bilinear  form
$$(( \ ,\, ))  : \BSplus(\mathcal{T}_c) \times  \BSplus(\mathcal{T}_c) \rightarrow \BOplus~.$$

If $c=0$, the Hopf pairing is the same as the pairing $\langle\ ,\
\rangle$ studied in \cite{BHMV1}. (In this case, the choice of clasp
does not matter.) As shown there, an orthogonal basis  for
$V_p(\mathcal{T})$ with respect to this pairing is $\{Q_0, Q_1,
\cdots Q_{d-1}\}~,$ where  $Q_0=1$, and $$Q_n=
\prod_{i=0}^{n-1}(z-\lambda_i)~.$$ Here $\lambda_i= -q^{i+1}-q^{-i-1}$,
and $z$ denotes the core of
the solid torus  with color $1$.  (In our current notation, $z$ is
$L_{0,1}$ and its framing is as drawn, {\em i.e.,} the blackboard framing.) The product
uses the commutative algebra structure on $V_p(\mathcal{T})$ where
$z^n$ means $n$ parallel copies of $z$. 

We now generalize this to the case when $c\neq 0$. Observe that
$V_p(\mathcal{T}_c)$ is a module over $V_p(\mathcal{T})$.
 Since $V_p(\mathcal{T})$ is commutative, we may
view this both as a left and as a right module structure, whatever is
notationally more
convenient. For example, one has  
 \[ L_{c,n} z \ =\ 
 zL_{c,n}\ =  \ \
 \ \begin{minipage}{0.8in}\includegraphics[width=0.8in]{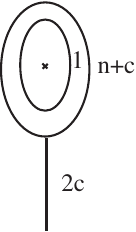}\end{minipage}\
 \ \  \] where the right hand side can then be expressed as a linear
 combination of $L_{c, n+1}$ and $L_{c, n-1}$ in the usual way, see \cite{MV,KL}.

\begin{de} For $0\leq n\leq d-c-1$, we set $$Q_n^{{(c)}}=
  L_{c,0} \prod_{i=c}^{c+n-1}(z-\lambda_i)~,$$ and ${Q}_n^{\prime (c)}= 
(\{n\} !)^{-1}
 Q_n^{(c)}$.
\end{de}

Here is the main result of this section. 

\begin{thm} \label{orthog} $\{
Q_0^{\prime (c)}, Q_1^{\prime (c)}, \cdots Q_{d-c-1}^{\prime (c)} 
\}$ is an orthogonal basis for $\BSplus(\mathcal{T}_c)$ with respect
to the Hopf pairing. Moreover 
\begin{equation}\label{Qwithprime}
(({Q}_n^{\prime (c)},{Q}_n^{\prime (c)}))= q^{ -c(c+1)/ 2} \frac {\{2c+2n+1\} !!\{ 2c+n+1\}^+ !}{\{n\} !}
\frac {(\{c\}_q !)^2} {\{1\}_q \{2c\}_q ! } 
\end{equation}
\end{thm}

We use the notation  $x\sim y$ to mean that $x=uy$
for some unit $u$ in $\BO$. Here $x,y$ can be numbers, or elements of
some module such as $V_p( \mathcal{T}_c)$. For $0 <n < p$, one has
both 
$\{n\}\sim h$ and  $\{n\}_q\sim h$,  and (hence)  $\{n\}^+= \{n\}_q / \{n\}\sim 1$. The following corollary is easily checked.

\begin{cor}\label{self}   $(({Q}_n^{\prime (c)},{Q}_n^{\prime (c)}))\sim h^c$. 
\end{cor}

\begin{proof}[Proof of Theorem \ref{orthog}] First, we show that the
  ${Q}_n^{\prime (c)}$ are a basis for $\BSplus(\mathcal{T}_c)$. The
  basis of $\BSplus(\mathcal{T}_c)$ constructed in \cite{GM} is
  $\{L_{c,0} v^n \, |  0\leq n\leq
  d-c-1\}$ where $v=h^{-1}(z+2)\in \BSplus(\mathcal T)$. Note  that
  $z-\lambda_i= (z+ 2)- (2+\lambda_i ) $ and $2+\lambda_i \sim h^2$.
Thus $h^{-1}(z- \lambda_i ) = v+ \text{(a scalar in $\BOplus$)}$. Thus
there is a triangular basis change, with ones on the diagonal, between the
basis $L_{c,0} v^n$ given in \cite{GM} and the vectors
  $h^{-n}{Q}_n^{(c)}$, which therefore form a basis. Since $\{n\}
  !\sim h^n$, it follows that the ${Q}_n^{\prime (c)}$ are also a
  basis for $\BSplus(\mathcal{T}_c)$.
  
  \begin{figure}[h]
\includegraphics[width=1.8in]{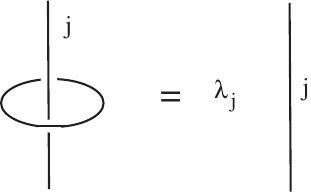}
\caption{Encircling a strand colored $j$ with a loop colored 
$1$ 
 has the same  
effect as multiplying by $\la_j$.} \label{encircle}
\end{figure}

Next, we show that this basis is orthogonal with respect to the Hopf
pairing. It is enough to show that the ${Q}_n^{ (c)}$ are
orthogonal. The idea of proof is the same as showing orthogonality of
the  $Q_n$ in \cite{BHMV1}.
Each $Q_n^{{(c)}}$ is a
  linear combination of $L_{c,m}$ where $0 \le m\le n$, and the coefficient
  of $L_{c,n}$ is one. We have that $((Q_{n+k}^{{(c)}}, Q_{n}^{{(c)}}))$
is zero when $k >0,$ as  $Q_{n+k}^{{(c)}}$  contains a $z-\lambda_{c+m}$ factor for $0\le m\le n+k-1$ which we may use to annihilate the  $L_{c,m}$ terms in $Q_{n}^{{(c)}}.$ 
See Figure \ref{encircle}.
Since the form is symmetric, this shows that the basis is
orthogonal. We remark that since $((\ ,\ ))$ is non-degenerate, this argument also shows that 
$Q_{n}^{{(c)}}=0$ for $n \geq d-c$, which will be used in the proof of
Theorem~\ref{mainth}.   
If $k=0$, we can eliminate all but the $L_{c,n}$ term in the expansion
of
 one copy of 
 $Q_{n}^{{(c)}}$ when computing $((Q_{n}^{{(c)}}, Q_{n}^{{(c)}}))$.
Thus
\[((Q_{n}^{{(c)}}, Q_{n}^{{(c)}}))=((Q_{n}^{{(c)}}, L_{c,n}))= ((L_{c,0}, L_{c,n}))
\prod_{i=c}^{c+n-1} (\lambda_{c+n}-\lambda_{i })~.
\]
But $((L_{c,0}, L_{c,n}))$ can be simplified using the first formula in the proof of \cite[Lemma(4.1)]{M1}. One gets \[((L_{c,0}, L_{c,n}))= C_{c+n, c,c} 
\left\langle 
\begin{matrix}
2c & c+n &c+n \\
n &c &c
\end{matrix}
\right\rangle 
\]
where $C_{c+n, c,c}$ is given in \cite[p. 550]{M1}, and the other
symbol is a tetrahedral symbol in the notation of \cite{MV}. The
formula for a tetrahedral symbol is a sum of products, but in this
case there is only one term in this sum.  Simplification gives 
\begin{equation}\label{Qwithoutprime}
((Q_n^{(c)},Q_n^{(c)}))= q^{ -c(c+1)/ 2} \{n\} !\{2c+2n+1\} !!\{ 2c+n+1\}^+ !
\frac {(\{c\}_q !)^2} {\{1\}_q \{2c\}_q ! } 
\end{equation} This easily implies (\ref{Qwithprime}) and completes
the proof. 
\end{proof}

\section{Orthogonal lollipop basis in higher genus}

In this section, we define the Hopf pairing and construct an
orthogonal basis for $\BSplus(\Si)$ where $\Si$ has arbitrary
genus $g$. The reader interested only in the one-holed torus case may
proceed directly to Section \ref{ptorus}.

{}For the remainder of this section, we assume the reader is familiar with the concepts
and notations of \cite[p.\,820-822]{GM} where a basis of $\BS(\Si)$ is defined in terms of certain colorings
of a lollipop tree. The same
construction actually yields a basis of $\BSplus(\Si)$, see
\cite[p.\,837]{GM}. 
 
Suppose $\Si$ is  equipped with a set $X$ of colored banded points. Fix a handlebody $H$ with boundary $\Si$, and a banded lollipop
tree $G$ in $H$ with boundary vertices equal to $X$, so that $H$ retracts onto $G$. The
{\em small graph basis} $\BG$ of $V_p(\Si)$ consists of the elements $\bg(a,b,c)$
represented by {\em small colorings} of $G$ \cite[Prop.~3.2]{GM}. See Fig.~\ref{figlolli} for an
example representing the basis element
$\bg((a_1,a_2,a_3),(b_1,b_2,b_3),(c_1,c_2,c_3,c_4))$ of $V_p(\Si)$,
where $\Si$ has genus $3$ and $|X|=2$, with colors $c_3$ and $c_4$ on
the two banded points.  The basis $\BB=
\{\bb(a,b,c) \}$
of $\BSplus(\Si)$ given in \cite[p.\,822]{GM} is indexed by the same
set of 
small colorings. It is obtained from $\BG$ by a triangular (but not
unimodular!)  basis
change. We will show how to modify the basis $\BB$ to make it orthogonal
with respect to the Hopf pairing which we discuss next.

\begin{figure}[ht]
\includegraphics[width=3in]{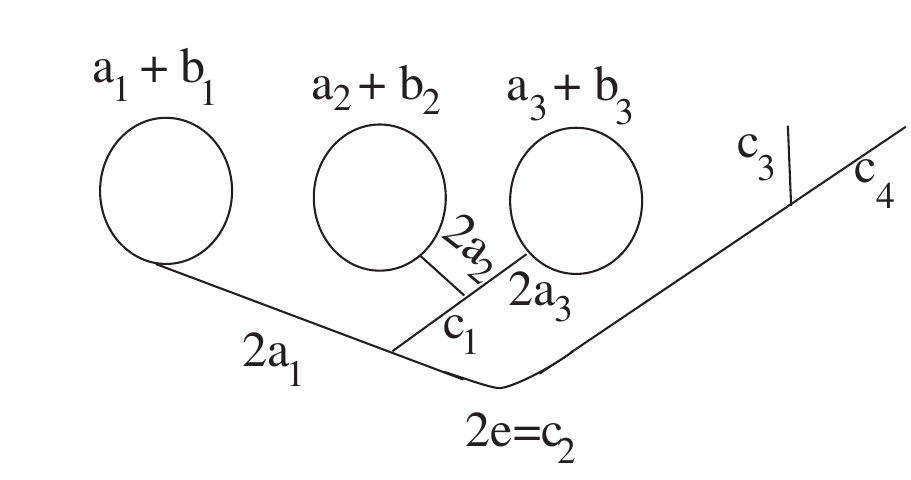}
\caption{Lollipop graph  $G$ in good position, with some coloring.} \label{figlolli}
\end{figure}

In order to define the Hopf pairing in our context, assume that $G$
and $H$ have been isotoped so that $G$ is presented by a half-plane
diagram with no crossings,  with the banded points on the boundary of
the half-plane, and one edge from each loop near the boundary of the
half-plane.  The
graph in 
Fig.~\ref{figlolli} is already drawn in this way. 
One may produce a dual $G'$ by rotating the half plane $180\,^{\circ}$
around its boundary line and then  bringing the loops together and
inserting clasps (linking pairs of loops) as illustrated in
Fig.~\ref{figlolli2}. Gluing $G$ and $G'$ along their boundary
vertices produces a banded trivalent graph which we denote by
$GG'$. Here, as in the genus one case, we insist that all the clasps be as drawn.

\begin{figure}[h]
\includegraphics[width=3in]{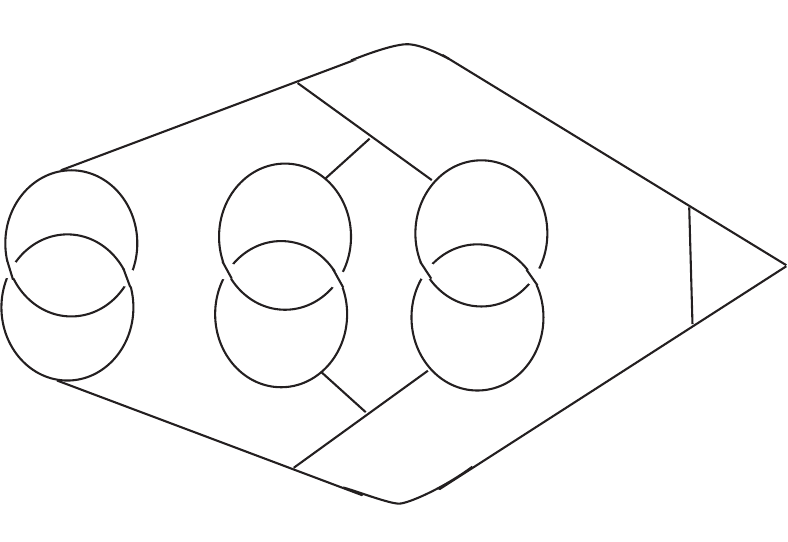}
\caption{$GG'$ for the colored graph in Fig.~\ref{figlolli}, with colors omitted.} \label{figlolli2}
\end{figure} 

As in the genus one case, although the Hopf pairing can be defined in
a basis-independent way, it is convenient to define it by its value on
basis elements, as follows: $((\bg(a,b,c),\bg(a',b',c'))$ is the evaluation of
$GG'$ with colors $(a,b,c)$ on the $G$ part of $GG'$, and colors
$(a',b',c')$ on the $G'$ part, where we use the $180\,^{\circ}$
rotation to transport $(a',b',c')$ from $G$ to $G'$. As before, this
restricts to a symmetric bilinear form 
\[ (( \ ,\, ))  : \BSplus(\Si) \times  \BSplus(\Si) \rightarrow \BOplus~. \]

The basis $\BB$ of $\BSplus(\Si)$ found in \cite[p. 822]{GM} is given
as follows: 
$$\bb(a,b,c) = h^{- \lfloor \frac 1 2 ( -e +\sum_j a_j )
  \rfloor }\bg(a,0,c) \prod_{j=1}^g \frac {(2+z_j)^{b_j}} {h^{b_j}}~.$$ Here $2e$ is
the color of the trunk edge of $G$, see \cite[p.\,821]{GM}, and $z_j$
is a zero-framed circle colored $1$ around the $j$-th hole. The
following modification will be our orthogonal basis:
\begin{equation} 
\tilde\bb(a,b,c) = h^{- \lfloor \frac 1 2 ( -e +\sum_j a_j )
  \rfloor }\bg(a,0,c) \prod_{j=1}^g \frac
{\prod_{i=a_j}^{a_j+b_j-1}(z_j-\lambda_i)}{\{b_j\}!}~.
\end{equation}

\begin{thm}\label{higher} The set $\widetilde \BB=\{ \tilde\bb(a,b,c)\} $ where $(a,b,c)$ runs through
  the small colorings of $G$, is a basis of $\BSplus(\Si)$ which is
  orthogonal with respect to the Hopf pairing. Moreover, 
$$
((\tilde\bb(a,b,c),\tilde\bb(a,b,c)))= h^{-2 \lfloor \frac 1 2 (-e+\sum_i a_i)\rfloor } \langle DG(a,c)\rangle     \prod_{i=1}^{g}
 \frac{ (({Q}_{b_i}^{\prime (a_i)},{Q}_{b_i}^{\prime (a_i)}))\{1\}_q}
 {\{2a_i+1\}_q} ~,
$$
 where $2e$ is
the color of the trunk edge, $\langle DG(a,c)\rangle$ is the evaluation of the colored
 banded graph $DG(a,c)$ defined below, and $(({Q}_{b_i}^{\prime
   (a_i)},{Q}_{b_i}^{\prime (a_i)}))$ is given in Eq.~(\ref{Qwithprime}).
\end{thm}

The proof is the same as in the one-holed torus case. To define  $DG(a,c)$, assume that $G(a,b,c)$ is a colored lollipop
graph in good position, as pictured in Fig.~\ref{figlolli},
representing the small graph basis element $\bg(a,b,c)$. Then  $DG(a,c)$
is the  colored graph in the plane obtained by removing the loop edges and then adjoining a reflected copy, as in Fig.~\ref{figlolli3}.

\begin{figure}[h]
\includegraphics[width=3in]{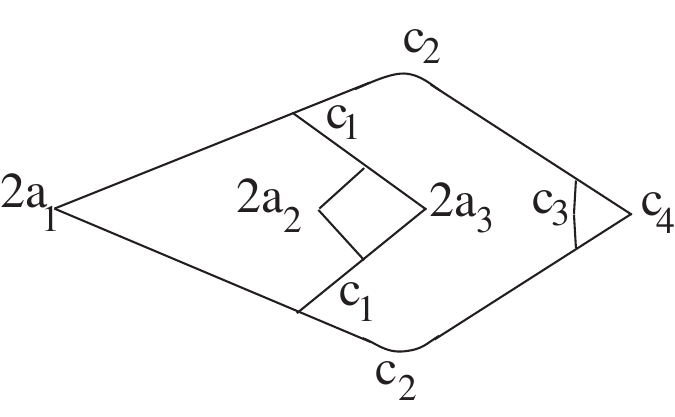}
\caption{$DG(a,c)$ for the colored graph in Fig.~\ref{figlolli}.} \label{figlolli3}
\end{figure}

Using formulas in \cite{MV}, it is easy to see that $\langle
DG(a,c)\rangle$ is always a unit in $\BOplus$. Therefore one has the
following 
\begin{cor}\label{exp} 
$$((\tilde\bb(a,b,c),\tilde\bb(a,b,c)))\sim \left\{
\begin{array}{cl}
h^{e} &\text{if \ $\sum_j a_j \equiv e \pmod{2}$} \\
h^{e+1}&\text{if \ $\sum_j a_j \not\equiv e \pmod{2}$}\\
\end{array}\right.$$
\end{cor}

\begin{rem}{\em The dual lattice $\BSs(\Si)$ and its refined
    version $\BSps(\Si)$ are defined in \cite[Def.\,8.1]{GM}
    and \cite[Sec.\,13]{GM}, respectively, using certain hermitian
    forms.  It is not hard to see that
    $\BSps(\Si)$ is also the lattice dual to $\BSplus(\Si)$
    with respect to the Hopf pairing: 
$$\BSps(\Si)= \{x\in 
\BSplus(\Si)\otimes \BQ(\zeta_p)\ 
|\ ((x,y))\in \BOplus=\BZ[\zeta_p]\text{ for
 all } y\in \BSplus(\Si)\}~.$$
}\end{rem}

Using orthogonality of the basis $\widetilde \BB$,  
Cor.~\ref{exp} easily implies that 

\begin{cor} The following rescaling 
\begin{align*}
\tilde\bb^\sharp(a,b,c)&=\left\{
\begin{array}{rl}
h^{-e}\,\tilde\bb(a,b,c) &\text{if \ $\sum_j a_j \equiv e \pmod{2}$} \\
h^{-e-1}\,\tilde\bb(a,b,c) &\text{if \  $\sum_j a_j \not\equiv e \pmod{2}$}\\
\end{array}\right.\\
&= h^{- \lceil \frac 1 2 ( e +\sum_j a_j )
  \rceil }\bg(a,0,c) \prod_{j=1}^g \frac
{\prod_{i=a_j}^{a_j+b_j-1}(z_j-\lambda_i)}{\{b_j\}!}
\end{align*} is a basis $\widetilde \BB^\sharp=\{
\tilde\bb^\sharp(a,b,c)\} $ of $\BSps(\Si)$ which is
  orthogonal with respect to the Hopf pairing.
\end{cor}

\begin{rem}\label{appl1}{\em In the above, we deduced the bases $\widetilde \BB$ of
    $\BSplus(\Si)$ and  $\widetilde \BB^\sharp$ of $\BSps(\Si)$ from
    the basis $\BB$ constructed in \cite{GM}. Here is a sketch of a
    more direct proof. First, observe that both $\widetilde \BB$ and $\widetilde
    \BB^\sharp $ are bases of $V_p(\Si)$. Now use the $3$-ball lemma and the lollipop lemma
    of \cite{GM} to show that $\widetilde \BB \subset \BSplus(\Si)$
    and $\widetilde \BB^\sharp \subset \BSps(\Si)$. Then use
    orthogonality and 
    Cor.~\ref{exp} to show that 
$$((\tilde\bb(a,b,c), \tilde\bb^\sharp(a',b',c'))\sim \delta_a^{a'}\delta_b^{b'}\delta_c^{c'}$$
and deduce that $\widetilde \BB$ generates
    $\BSplus(\Si)$ and  $\widetilde \BB^\sharp$ generates
    $\BSps(\Si)$. This proof avoids the index counting argument of
    \cite{GM}.
}\end{rem} 

\begin{rem}\label{appl2} 
{\em
Recall $\BJ_p^+(N)$, the Frohman Kania-Bartoszynska ideal \cite{FKB,GM,G4} of a 
 compact connected oriented
3-manifold $N$ with connected boundary $\Si$.  
Using our orthogonal basis, we can give generators for this ideal as
follows.  The TQFT associates to $N$ a vector $[N]\in \BSplus(\Si)$
which is well defined up to multiplication by a root of unity. Write
$$ [N]=\sum_{(a,b,c)}  x_{(a,b,c)} \tilde\bb(a,b,c)~,$$ 
where the coefficients $x_{(a,b,c)}$ lie in $\BOplus$, and the sum is over small admissible colorings of a lollipop graph
in a handlebody with boundary $\Si.$ 
 Then by 
 Theorem \ref{higher} and Corollary
\ref{exp}, 
together with \cite[Thm.~16.5]{GM},  one has that $\BJ_p^+(N)$ is the
$\BOplus$-ideal generated by 
the numbers
 $h^{\varepsilon( a)} x_{(a,b,c)}$
where 
 $\varepsilon(a)$ 
is zero or one accordingly as  $\sum_j a_j$ is even or odd.}
\end{rem}

\section{Matrices for $t$ and $t^*$}\label{ptorus}

Recall that elements of $V_p({\mathcal
  T}_c)$ are represented skein-theoretically as linear combinations of
colored banded graphs in a solid torus which meet the boundary nicely
in the colored banded point. We refer to such linear combinations
as skein elements. 

\begin{de} We let $t$ be the endomorphism of
$V_p({\mathcal T}_c)$ induced by a full positive twist of the solid
torus. It follows from \cite{GM} that $t$ preserves the lattice $\BSplus({\mathcal T}_c)$. 
\end{de} In the basis (\ref{Lbasis}) of
$V_p({\mathcal T}_c)$, $t$ is diagonal:
\begin{equation} \label{mu} t(L_{c,n})=\mu_{c+n} L_{c,n}~,\end{equation} where $\mu_k=(-1)^k
A^{k(k+2)}$ is the twist eigenvalue \cite{BHMV1}.

Let $\omega_+\in V_p({\mathcal
  T})$ be defined as 
\begin{equation}\label{oplus} 
\omega_+= \sum_{m=0}^{d-1}  \gamma_m \frac {Q_m} { \{m\}!}  \ \text{ where  }
\gamma_m= \frac { (-A)^{(m^2+5m)/ 2}
}{\prod_{k=1}^{m}(A^{2k+1}-1)} \sim 1 .
\end{equation}
(To see that $\gamma_m \sim 1$ one uses that $1-A$ is a unit in
$\BOplus$, see \cite[Lemma 4.1.(i)]{GMW}.)
 It follows from Theorem \ref{orthog} (for $c=0$) that $\omega_+\in \BSplus({\mathcal T})$.

\begin{thm} If $x\in V_p({\mathcal T}_c)$ is represented by a
  skein element (which we also denote by $x$), then $t(x)$ is represented by the
  union of $x$ with a zero-framed $\omega_+$ placed on a meridian of
  the solid torus pushed slightly into the interior. (See Fig.~\ref{omtwist}.)
\end{thm}

\begin{figure}[h]
\includegraphics[width=1.8in]{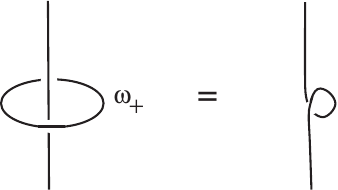}
\caption{Encircling a strand with $\omega_+$ has the same
effect in TQFT as giving that strand a positive twist.} \label{omtwist}
\end{figure}

\begin{proof} The existence of an $\omega_+$ with this property is 
  well-known in TQFT \cite{BHMV2}. The expression (\ref{oplus}) for it
  can be deduced from \cite{BHMV1}, as follows. Let
  $\omega_{\text{BHMV1}}$ be the skein element of \cite{BHMV1}\footnote{Warning: the $\omega$ of \cite{BHMV1} is not the same as the one of
\cite{BHMV2}.}. By \cite[4.3(2)]{BHMV1}, encircling with
$t(\omega_{\text{BHMV1}})$ produces a {\em negative} twist, and 
\[t(\omega_{\text{BHMV1}})=  \sum_{m=0}^{d-1} \frac {(-1)^m A^{-2m}}{\prod_{j=2}^{2m+1}(A^j-1)} \mu_m Q_m=
\sum_{m=0}^{d-1} \frac { A^{m^2} }{\prod_{j=2}^{2m+1}(A^j-1)}  Q_m  ~.\]
A skein expression for $\omega_+$ can now be obtained from this
formula for
$t(\omega_{\text{BHMV1}})$ 
by
taking mirror
images of the skein elements and conjugating the scalars (using
$\overline{A}=A^{-1}$). (Note that this operation leaves $Q_m$
invariant.) Formula (\ref{oplus}) follows.
\end{proof}

\begin{de} We let $t^*$ be the endomorphism of
$V_p({\mathcal T}_c)$ given by $$t^*(x) = x\, \omega_+ ~.$$ It follows
from \cite{GM} that $t^*$ preserves $\BSplus({\mathcal T}_c)$ (since
$\omega_+\in \BSplus({\mathcal T})$).
\end{de} Here the multiplication on the right hand side is the module
structure discussed in Section~\ref{sec2}. Observe
that $t^*(x)$ is represented by the
  union of $x$ with a zero-framed $\omega_+$ placed on the standard longitude  of
  the solid torus pushed slightly into the interior. (See the last
  figure in Section~\ref{sec2}.)

Recall that the mapping class group $\Gamma_{1,1}$ is generated by the
two Dehn twists about the meridian and the longitude of the torus. It
follows from standard results in the skein-theoretical construction of
the TQFT-representation $\rho_p$  \cite{BHMV2} that $t$ and $t^*$ are
$\rho_p$ of certain lifts of these twists to the extended mapping
class group. Thus we have the following 

\begin{thm} The endomorphisms $t$ and $t^*$ describe the TQFT
  representation $\rho_p$ of the extended mapping class group on
  $\BSplus({\mathcal T}_c)$.
\end{thm} 

\begin{rem}\label{rem45} {\em The lifts chosen here are the
`geometric' lifts as discussed in \cite{MRext}. Choosing different lifts would multiply $t$ and
$t^*$ by well-known scalar factors. 
Adapting the techniques of \cite{MRext}, one
can show that $t t^* t=t^* t t^*$ and
\begin{equation}\label{relcomp}(tt^*t)^4=q^{-6+2c(c+1)-p(p+1)/2}~.
\end{equation} 
We will discuss  
our 
method to do computations like these in a more general
context in \cite{GMnew}.} 
\end{rem}

 We now set out to compute matrix coefficients for $t$ and $t^*$. The following
proposition tells us that it will be good to use the basis constructed
in Section~\ref{sec2} which is
orthogonal for the Hopf pairing.
\begin{prop} \label{adjoint} The endomorphisms $t$ and $t^*$ are
  adjoint for the Hopf pairing:  
\[ (( t(x) ,y )) = (( x , {t}^\star( y) ))~. \]
\end{prop}
\begin{proof} We need to show that 
$$((  x \cup  \omega_+ \text{(along the meridian)}, y)) = ((x, y \cup
\omega_+ \text{(along the longitude}))~.$$ But this is clear, since in
the standard decomposition of $S^3$ into two solid tori, the meridian of one solid torus is the longitude of the other. 
\end{proof}

It follows that if we denote the matrix coefficients for $t$ and $t^*$ in the basis
$(Q^{\prime (c)}_n)$ by $a^{(c)}_{m,n}$ and $b^{(c)}_{n,m}$,
respectively:
 \[t(Q_n^{\prime (c)})= \sum_{m=0}^{d-1-c} a^{(c)}_{m,n} {Q}_m^{\prime (c)}\]
\[t^{\star}({Q}_m^{\prime (c)})= \sum_{n=0}^{d-1-c} b^{(c)}_{n,m}
{Q}_n^{\prime (c)}~,\] then we have that:
 \begin{align} 
\notag a^{(c)}_{m,n}(({Q}_m^{\prime (c)},{Q}_m^{\prime (c)}))
&= ((t({Q}_n^{\prime (c)}),{Q}_m^{\prime (c)}))\\ 
\notag
 	= (({Q}_n^{\prime (c)},t^\star({Q}_m^{\prime (c)})))
&=b^{(c)}_{n,m}(({Q}_n^{\prime (c)},{Q}_n^{\prime (c)})).
	 \end{align}
	
Hence 
\begin{equation} \label{anbn} a^{(c)}_{m,n} = \frac {(({Q}_n^{\prime (c)},{Q}_n^{\prime
    (c)}))}{(({Q}_m^{\prime (c)},{Q}_m^{\prime (c)}))} b^{(c)}_{n,m}
= R^{(c)}_{n,m} 
b^{(c)}_{n,m},
\end{equation}

where  
\begin{equation} \label{Rmn} R^{(c)}_{n,m}= 
\frac {\{m\}!\{2c+2n+1\} !!\{ 2c+n+1\}^+ !}{{\{n\}!\{2c+2m+1\} !!\{ 2c+m+1\}^+ !}} \sim 1~.
\end{equation}

That $R^{(c)}_{n,m} \sim 1$ follows easily from the remarks before Corollary \ref{self}.
These same remarks and  (\ref{oplus})  justify the $\sim$'s appearing in (\ref{Clmn}) and Theorem \ref{mainth} below.

Thus we are reduced to compute the $b^{(c)}_{n,m}$. To do so, we
need the coefficients defined in the following proposition, which
describes the module structure of $V_p(\mathcal{T}_c)$ over the
algebra $V_p(\mathcal{T})$ in terms of the $Q$-bases.

\begin{prop}\label{4.7}  
\begin{equation} \label{form47}Q_m Q^{(c)}_n= \sum_{l=0}^{\min\{m,n+c\}} C^l_{m,n+c}
Q^{(c)}_{m+n-l}\end{equation}
where 
\begin{equation} \label{Clmn}C^l_{m,n}= (-1)^l \frac { \{m\}!\{n\}!\{m+n+1\}!}{
  \{m-l\}!\{n-l\}!\{m+n+1-l\}!\{l\}!} \sim h^{2l}~. 
\end{equation}
\end{prop}

We defer the proof to the end of this section. We are now ready to
state the main result of this section.

\begin{thm}\label{mainth} The matrix for $t^*$ is lower triangular, and the matrix
  for $t$ is upper triangular: 
$$
b^{(c)}_{n,m}=a^{(c)}_{m,n}=0 \text{\
    if\ } m > n~.
$$  Moreover, if $m\leq n$, then  
$$ b^{(c)}_{n,m}= \sum_{l=0}^{m+c} b^{(c)}_{n,m,l}, \ \ a^{(c)}_{m,n}=
\sum_{l=0}^{m+c} a^{(c)}_{m,n,l}~,$$ where
 \begin{align} 
\notag b^{(c)}_{n,m,l} &=  C^{l}_{l+n-m,m+c} \, 
\frac {\gamma_{l+n-m}}{\{l+n-m\}!}
\frac {\{n\}!}
{\{m\}!}\sim h^{l}\\ 
\notag a^{(c)}_{m,n,l} &=  C^{l}_{l+n-m,m+c} \, 
\frac {\gamma_{l+n-m}}{\{l+n-m\}!}
\frac {\{n\}!} {\{m\}!}\,
R^{(c)}_{n,m}\sim h^{l} ~.
\end{align}
\end{thm}

\begin{proof} We compute  
\begin{align}
\notag &t^{\star}({Q}_m^{\prime (c)})= \omega_+\, {Q}_m^{\prime (c)}= 
\sum_{k=0}^{d-1} \gamma_k \frac {Q_k}{\{k\}!} \frac {Q_m^{(c)}}{\{m\}!} \\
\notag &=
 \sum_{k=0}^{d-1}  \frac {\gamma_k }{\{k\}! \{m\}!}
 \sum_{l=0}^{\min\{k,m+c\}} C^l_{k,m+c}\, Q^{(c)}_{k+m-l}  \\
\notag &=\sum_{l=0}^{m+c}\,\,\sum_{k=l}^{d-1}  \frac {\gamma_k
}{\{k\}! \{m\}!} \, C^l_{k,m+c} \,Q^{(c)}_{k+m-l}  
\text{ (changing the order of summation)}\\
\notag &=\sum_{l=0}^{m+c}\,\,\sum_{n=m}^{d-1+m-l}  \frac {\gamma_{l+n-m}
}{\{l+n-m\}! \{m\}!} \, C^l_{l+n-m,m+c} \,Q^{(c)}_{n} \  \text{ (letting
  $n=k+m-l$)}\\
\notag &=\sum_{l=0}^{m+c}\,\,\sum_{n=m}^{d-1-c}  \frac {\gamma_{l+n-m}\{n\}!
}{\{l+n-m\}! \{m\}!} \, C^l_{l+n-m,m+c} \,Q^{\prime (c)}_{n} \  \text{ (as
  ${Q}^{(c)}_{n}=0$ for $n\ge d-c$)}\\
\notag &=\sum_{n=m}^{d-1-c}\,\, \sum_{l=0}^{m+c} b_{n,m,l}^{(c)}\,
Q^{\prime (c)}_{n} 
 \end{align}
where $b_{n,m,l}^{(c)}$ is as in the statement of the theorem. This
proves the result for  the
matrix coefficients $b_{n,m}^{(c)}$ of $t^\star$. To get the matrix
coefficients for $t$, it now suffices to apply Eq.~(\ref{anbn}).
 \end{proof}

It remains to give the 
\begin{proof}[Proof of Prop.~\ref{4.7}.] Let us first consider
the case $c=0$. It is easy to see that 
\begin{equation}\label{formulaQ} Q_m Q_n= \sum_{l=0}^{\min\{m,n\}}
C^l_{m,n} Q_{m+n-l}
\end{equation} for some coefficients $C^l_{m,n}$ in the range $0 \le l \le \min\{m,n\}.$ 
 We must show that
$C^l_{m,n}$ satisfies Eq.~(\ref{Clmn}). 
Observe that if (\ref{Clmn}) holds for
$n \le m$, then it holds in general, by symmetry. So we may assume $n
\le m$. We now fix $m$, and prove the formula by induction on $n$, for
all $0 \le l \le n$ at once. If $n=0$,  $C^0_{m,0} =1$, and formula
(\ref{Clmn}) agrees. Assume inductively the formula is proven for $n$,
and consider it for $n+1$. Again, we may assume that $n+1 \leq m$. 

To perform the inductive step, let  $\beta_{m,n}= \lambda_m-\lambda_n=
\{n-m\}\{m+n+2\}$.   Since $Q_{n+1}= (z-\lambda_n) Q_{n}$, after
multiplying (\ref{formulaQ}) by $z-\lambda_n$ we get the following
recursive relation: $$C^l_{m,n+1} =\left\{
\begin{array}{cl} 
C^0_{m,n} &\text{if \ $l=0$} \\
C^l_{m,n}+ \beta_{m+n-l+1,n} C^{l-1}_{m,n}  &\text{if \  $ 1\leq l\leq
  n$}\\
\beta_{m,n} C^{n}_{m,n}  &\text{if \  $ l=n+1$}
\end{array}\right.$$

It is not hard to check that formula (\ref{Clmn}) satisfies these
recursion relations. Thus, the $c=0$ case of Prop.~\ref{4.7}
is proved. 

For the general case, let $$Q_{n,c}=
\prod_{i=c}^{c+n-1}(z-\lambda_i)= \frac{Q_{n+c}} {Q_c} $$ so that
$Q^{(c)}_{n}= Q_{n,c} L_{c,0}.$ We have
\[Q_m Q_{n,c}= \sum_{l=0}^{\min\{m,n+c\}} C^l_{m,n+c}
Q_{m+n-l,c}~,\] as follows from applying (\ref{formulaQ}) with $n+c$ in
place of $n$, and dividing by $Q_c$. Multiplying both sides of this
equation by $L_{c,0}$, we get (\ref{form47}), completing the proof.  
\end{proof}

We remark that Prop.~\ref{4.7} and its proof also work in the context
of the Kauffman Bracket skein module of a solid torus (relative to a
$2c$-colored banded point on its boundary if $c>0$) over the Laurent
polynomial ring 
 $\BZ[A^{\pm1}]$
 where $A$ is generic ({\em i.e.} not necessarily a root
of unity). 

\section{Irreduciblity of the representation over $\BF_p$}\label{sec5}

Recall that 
$\BZ[\zeta_p]/(h)$ is the finite field
$\BF_p=\BZ/p\BZ$. The 
TQFT-representation
$\rho_p$ on $\BSplus({\mathcal T}_c)$ induces a
representation $\rho_{p,0}$ on $${\mathcal{S}}^+_{p,0}({\mathcal T}_c)= \BSplus({\mathcal T}_c)\slash h
\BSplus({\mathcal T}_c)~,$$ which is an $\BF_p$-vector space of
dimension $d-c$. (Recall $d=(p-1)/2$.) This representation is generated by the two
matrices $t$ and $t^*$ considered modulo $h$. Recall from
Remark~\ref{rem45} that $t t^* t=t^* t t^*$ and $(tt^*t)^4\equiv 1\pmod{h}$. Thus $\rho_{p,0}$ factors through a representation of
\begin{equation}\label{presSL}SL(2,\BZ) =<T, T^*\, | \, T T^* T=T^* T T^*\, ,\, (TT^*T)^4=1>~,\end{equation} 
  \[ \text{where} \quad T= \begin{bmatrix} 1 & 1\\ 0 & 1\end{bmatrix} \quad \text{and} \ T^*=
  \begin{bmatrix} 1 & 0\\ -1& 1 \end{bmatrix}. \] 
Note that  $T$ and $T^*$ are respectively the matrices for the action
of a meridinal and a longitudinal Dehn twist on the homology of the
torus $\mathcal{T}$  
 with respect to the meridian-longitude basis.

We now come to the main result of this section. 

\begin{thm}\label{5.1} 
For every  $0\leq c\leq d-1$, 
the representation $\rho_{p,0}$ factors through
  $SL(2,\BF_p)$. Moreover, the induced representation of $SL(2,\BF_p)$
   is
  isomorphic to the representation of $SL(2,\BF_p)$ on the vector
  space $H_{p,D}$ of homogeneous polynomials in two variables over
  $\BF_p$ of total degree $D$, where $D+1= d-c$.
\end{thm}
Here we consider the left  action of $SL_2( \BF_p)$ on 
polynomials in variables $x$ and $y$ over  $\BF_p$ given by
\[\begin{bmatrix} a & b\\ c & d
\end{bmatrix}  x^m y^n  =  (ax+cy)^m (bx+dy)^n.\]  

This restricts to a representation $r_{p,D}$ of $SL_2( \BF_p)$ on the vector space $H_{p,D}$ of homogeneous polynomials of a given degree $D$. The dimension of this representation is $D+1.$  It is
irreducible if and only if $0\leq D \leq p-1$
\cite[pp. 31-32]{H}. These are the only irreducible representations of
$SL_2(\BF_p)$ over $\BF_p$.

\begin{cor} The representation $\rho_{p,0}$ on the $\BF_p$-vector
  space ${\mathcal{S}}^+_{p,0}({\mathcal T}_c)$ is irreducible for all
  $0\leq c\leq d-1$. 
\end{cor}

\begin{rem}{\em The matrices $t$ and $t^*$ have order $p$, as follows
    from (\ref{mu}). But this is not enough to ensure that $\rho_{p,0}$ factors through
  $SL(2,\BF_p)$, because adding the relation $T^p=(T^*)^p=1$ to the
  presentation (\ref{presSL}) defines still an infinite group if
  $p\geq 7$ (it is a double cover of the $(2,3,p)$ triangle group). 
}\end{rem}

\begin{rem}{\em The reader may wonder whether  ${\mathcal{S}}^+_{p,0}(\Si)$
    is an irreducible representation of the mapping class group for
    general surfaces $\Si$. In fact, if the genus of
    $\Si$ is at least $3$, ${\mathcal{S}}^+_{p,0}(\Si)$ is not
    irreducible, as already observed in \cite[Section
    14]{GM}. 
The
situation can be understood completely
    using the  orthogonal
    lollipop basis: one finds that ${\mathcal{S}}^+_{p,0}(\Si)$ has a
    composition series with at most two irreducible pieces (at least
    when $\Si$ has at most one 
    colored point). Details
    will be given elsewhere  
    \cite{GM3}.

We remark that 
over
the complex numbers,
the representation $\BSplus(\Si)\otimes
\mathbb{C}$
 is always irreducible when $p$ is prime 
and    when $\Si$ has at most one 
    colored point \cite{GM3}. 
  If $\Si$ 
has no colored points, 
this result
can be obtained  by adapting
Robert's proof of the analogous
     result for the $SU(2)$-TQFT (where $A$ is a primitive 
$4p$-th root
     of unity) \cite{R}. 
Irreducibility over the complex numbers implies irreducibility
over the cyclotomic field  $\BQ(\zeta_p)$, but does not, of course,  imply 
    irreducibility over the finite field $\BF_p$. Also, Roberts'
    argument does not apply over $\BF_p$, because it uses
    the fact that the twist eigenvalues
    $\mu_k=(-A)^{k(k+2)}=\zeta_p^{(d+1)k(k+2)}$ for $0\leq k\leq d-1$ are all
distinct.
But these eigenvalues become 
 equal to $1$
in $\BF_p$ (since $\zeta_p=1$ in $\BF_p$).
}\end{rem}

\begin{proof}[Proof of Theorem \ref{5.1}.] Consider the basis 
$$ (x^{D-n} y^n)_{n=0,1,\ldots,D}$$ of $H_{p,D}$. In this basis, $T$
is upper triangular, $T^*$ is lower triangular, and the matrix coefficients of
  $T$ and $T^*$ are given by
 \begin{align*} T  (x^{D-n} y^{n})  &= \sum_{m=0}^n \alpha^{(D)}_{m,n}   x^{D-m} y^{m} \quad \text{where $\alpha^{(D)}_{m,n}= \binom{n}{m}$}~,\\ 
T^{\star}  (x^{D-n} y^{n})  &= \sum_{m=n}^D  \beta^{(D)}_{m,n}
x^{D-m} y^{m} \quad \text{where $\beta^{(D)}_{m,n}= (-1)^{m-n} \binom
  {D-n}{m-n}$}~. \end{align*}

The matrix coefficients of $t$ and $t^*$ acting on
${\mathcal{S}}^+_{p,0}({\mathcal T}_c)$ are given by  $a^{(c)}_{m,n}$
and  $b^{(c)}_{m,n}\pmod{h}$, which we denote by  $\hat a^{(c)}_{m,n}$
and  $\hat b^{(c)}_{m,n}$. Recall that they lie in the finite field $\BF_p$.

\begin{lem}\label{5.5}   \[\hat a^{(c)}_{m,n} = \frac {(-1)^{n-m} (2c+2n+1)!!}
  {(n-m)! (2c+2m+1)!!} 
\]
\[\hat b^{(c)}_{m,n}  =  (-2)^{n-m} \binom{m} {n}~.\]
\end{lem}

\begin{proof} This follows from Theorem~\ref{mainth}, using that $q\equiv 1 \pmod{h}$, $A\equiv -1 \pmod{h} $, $\gamma_m\equiv  (-1)^m 2^{-m} \pmod{h} $, $\{n\}^+\equiv  2 \pmod{h} $, and $\{n\}/\{1\}\equiv n \pmod{h} $. 
\end{proof}

Now assume  $D=d-c-1$ and define a vector
space isomorphism $\Phi:  H_{p,D} \rightarrow
{\mathcal{S}}^+_{p,0}({\mathcal T}_c)$ by 
\[\Phi(x^{D-n} y^n)= (-1)^n \frac  {n!} {(2c+2n+1)!!}
Q_n^{\prime (c)} .\] 

We claim that $\Phi$ intertwines the representations $r_{p,D}$ 
on $H_{p,D}$ and $\rho_{p,0}$ on ${\mathcal{S}}^+_{p,0}({\mathcal
  T}_c)$. To see this, it suffices to check that (in $\BF_p$) one has
\[ (-1)^m \frac {m! } {(2c+2m+1)!!}\alpha^{(d-c-1)}_{m,n}= (-1)^n \frac {n!}{ (2c+2n+1)!!}  \hat a^{(c)}_{m,n}\]
\[ (-1)^m \frac  {m! } { (2c+2m+1)!!} \beta^{(d-c-1)}_{m,n}= (-1)^n
\frac {n!}{ (2c+2n+1)!!}  \hat b^{(c)}_{m,n} ~.\]

The first is immediate. The second reduces to
\[(-2)^{-n} (d-c-1-n)!(2c+2n+1)!!= (-2)^{-m} (d-c-1-m)!(2c+2m+1)!!\]
which  
is true due to the following lemma.

\begin{lem} For $0 \le k \le d-1$,
\[(-2)^{-k} (d-k-1)! (2k+1)!!= (d-1) ! \pmod{p}~.\]
\end{lem}

\begin{proof} We prove by induction on $k$. Let $u_k$ be the
  expression on the left hand side.
We have that $u_0=(d-1) !$, and also $u_{k+1}= (2k+3)u_k/(-2)(d-k-1)= u_k \in \BF_p$, as $(-2)(d-k-1)=2k+3\in \BF_p$.
\end{proof}
 
This completes the proof of Theorem \ref{5.1}.
\end{proof}

 \begin{rem} {\em With van Wamelen  
   \cite[p.~264]{GMW},
 we gave a
     formula for the action of $t$ on  $v^n$  in the Kauffman bracket 
skein
module of $S^1 \times D^2$.
This formula gives an expression for the matrix coefficients $\hat
a_{m,n}^{(0)}$ which looks quite different (and more complicated) than
the $c=0$ case of the expression in
Lemma \ref{5.5}. Comparing the two expressions 
leads to a proof of 
the following 
identity:

 For integers $n$, $m$, and $i$ with $n \ge m \ge 1$, and $i \geq 0$,
\[
\sum_{k=m}^{n} (-1)^{k+n}\frac k n \binom{k^2-1}{n-m-i} \binom{2n}{n-k} \binom{k+m-1}{k-m}=
\left\{
\begin{array}{cl}
\displaystyle{ \frac {(2n-1) !}  {(n-m)!(2m-1) !}}&\text{if \ $i=0$}\\
 0 &\text{if \ $i>0$.} \\
\end{array}\right.
\]

Conversely, if one has an independent proof of this identity for
$i=0$, then one can deduce the
formula for $\hat a_{m,n}^{(0)}$ in Lemma \ref{5.5} from the formula for $t(v^n)$ in  \cite{GMW}.

We thank C. Krattenthaler \cite{K} for providing such an independent
proof of the 
identity. Although the sum is not hypergeometric (due to the $k^2$ in
the binomial coefficient), Krattenthaler was able to reduce the proof to 
applications of Dixon's identity \cite[(2.3.3.6); Appendix (III.9)]{S}.
}\end{rem}

\end{document}